\documentclass{amsart}
\usepackage{amssymb,amsfonts,latexsym}

\newtheorem{theorem}{Theorem}[section]
\newtheorem{lemma}[theorem]{Lemma}
\newtheorem{proposition}[theorem]{Proposition}

\theoremstyle{definition}

\newtheorem{remark}[theorem]{Remark}
\newtheorem{general remarks}[theorem]{General remarks}

\newcommand{\ben}{\begin{enumerate}}
\newcommand{\een}{\end{enumerate}}

 \newcommand{\bee}[1]{\begin{equation}\label{#1}}
 \newcommand{\ene}{\end{equation}}

\newcommand{\I}{\mbox{Id}}

\begin{document}

\title[Multialternating graded polynomials and growth]{Multialternating graded polynomials and growth of polynomial identities}
\author{Eli Aljadeff}
\address{Department of Mathematics, Technion-Israel Institute of
Technology, Haifa 32000, Israel} \email{aljadeff@tx.technion.ac.il}
\author{Antonio Giambruno}
\address{Dipartimento di Matematica e Informatica,
Universit\`a di Palermo, Via Archirafi 34, 90123 Palermo, Italy}
\email{antonio.giambruno@unipa.it}

\date{}

\keywords{graded algebra, polynomial identity, growth, codimensions}

\subjclass[2010]{Primary 16R50, 16P90, 16R10, 16W50}

\thanks{The first author was supported by the ISRAEL SCIENCE FOUNDATION
(grant No. 1283/08) and by the E. SCHAVER RESEARCH FUND. The second
author was partially supported by MIUR of Italy}

\begin{abstract} Let $G$ be a finite group and $A$ a finite dimensional
$G$-graded algebra over a field of characteristic zero.
When $A$ is simple as a $G$-graded algebra, by mean of Regev central polynomials
we construct multialternating graded
polynomials of arbitrarily large degree non vanishing on $A$.
As a consequence we compute the exponential rate of growth of the
sequence of graded codimensions of an arbitrary $G$-graded algebra satisfying an
ordinary polynomial identity. If $c_n^G(A), n=1,2,\ldots$, is the sequence of graded
codimensions of $A$,  we prove that
$exp^G(A)=\lim_{n\to\infty}\sqrt[n]{c_n^G(A)}$, the $G$-exponent of $A$,
exists and is an integer.
This result was proved in \cite{AGL} and \cite{GL}, in case $G$ is abelian.
\end{abstract}

\maketitle

\section*{Introduction}

Let $F$ be a field of characteristic zero and $A$ an $F$-algebra graded by a finite group $G$.
We shall assume throughout that $A$ is a PI-algebra, i.e., it satisfies an
ordinary polynomial identity.
The graded polynomial identities satisfied by $A$ and their growth have been
extensively studied in the last years in an effort to develop a general theory
that generalizes the theory of ordinary polynomial identities.

For instance in analogy with a basic result in Kemer's theory (see \cite{kemer}),
it was recently proved in  \cite{AB}
(and independently in \cite{S} for $G$ abelian) that if $A$ is a finitely
generated algebra, then $A$ satisfies the same graded identities as a finite
dimensional graded algebra. As a consequence one can reduce the study of
an arbitrary $G$-graded PI-algebra to that of the Grassmann envelope of a finite dimensional
$\mathbb{Z}_2\times G$-graded algebra.

Also, in analogy with the results in \cite{gz1} and \cite{gz2} concerning the existence of the exponent of
a PI-algebra, in \cite{AGL} and \cite{GL} it was shown that when $G$ is an abelian group
$exp^G(A)$, the graded exponent of $A$, exists and is an integer.

Here we focus on the growth of the graded codimensions of $A$. Recall that if $P_n^G$
is the space of multilinear graded polynomials of degree $n$  and $\mbox{Id}^G(A)$ is the ideal
of graded identities of $A$, then $c_n^G(A)$, the $n$th $G$-codimension of $A$,
measures the dimension of $P_n^G$ mod. $\mbox{Id}^G(A)$.

It was known that the sequence $c_n^G(A)$, $n=1,2,\ldots,$ is exponentially
bounded (\cite{gia-reg}), but only recently its exponential rate of growth was captured
in case $G$ is an abelian group.
In fact in \cite{AGL} for finitely generated algebras and in \cite{GL}
in general, it was shown that if $A$ is any $G$-graded PI-algebra and $G$ is a finite
abelian group, then  $exp^G(A)=\lim_{n\to\infty}\sqrt[n]{c_n^G(A)}$ exists and is an integer
called the $G$-exponent of $A$.
Moreover the $G$-exponent can be explicitly computed and equals
the dimension of a suitable finite dimensional semisimple graded algebra related to $A$.

In this note we shall extend the above  result by proving that the $G$-exponent exists and
 is an integer in case of arbitrary (non necessarily abelian) groups $G$.
We notice that the proof given in \cite{GL} works also when $G$ is a non abelian
group modulo two basic ingredients that we manage to prove here.
The first result says that the multiplicities in the $n$th graded cocharacter of $A$ are
polynomially bounded, the second is of independent interest and consists on
the construction of suitable multialternating graded polynomials non vanishing in a
finite dimensional $G$-graded algebra which is simple as a graded algebra.

Throughout the paper $F$ will be a field of characteristic zero
and $A$ an associative $F$-algebra satisfying a non-trivial
polynomial identity.
We assume that $A$ is graded by a finite group $G=\{g_1=1,g_2, \ldots, g_s\}$
and we let $A=\oplus_{g\in G}A_g$ be its decomposition into a sum of
homogeneous components.

\section{Multialternating polynomials on $G$-simple algebras}

We start by introducing some standard notation.

We let $F\langle X, G\rangle$ be the free associative $G$-graded
algebra of countable rank over $F$. The set $X$ decomposes as
$X=\bigcup_{i=1}^s X_{g_i},$ where  the sets $X_{g_i}=\{x_{1,g_i},
x_{2,g_i}, \ldots \}$ are disjoint, and the elements of $X_{g_i}$
have homogeneous degree $g_i.$ The algebra $F\langle X, G\rangle$
is endowed with a natural $G$-grading $F\langle X, G\rangle= \oplus_{g\in
G}{\mathcal F}_{g}$, where ${\mathcal F}_{g}$ is the subspace
spanned by the monomials $x_{i_1,g_{j_1}}\cdots x_{i_t,g_{j_t}}$
of  homogeneous degree $g=g_{j_1}\cdots g_{j_t}$.

Recall that a graded polynomial $f \in F\langle X, G \rangle$ is a
graded (polynomial) identity of $A$ if $f$ vanishes under
all graded substitutions $x_{i, g} = a_{g}\in A_{g}$.
Let also $\I^G(A)=\{f\in F\langle X, G \rangle \mid \ f\equiv 0 \
\mbox{on}\ A\}$ be the ideal of graded identities of $A.$

We say that an algebra $A$ is $G$-graded simple if $A$ is a $G$-graded
algebra which is simple as a graded algebra.

Let $A$ be a finite dimensional $G$-graded simple algebra over
an algebraically closed field of characteristic zero.
The purpose of this section it to produce non identity $G$-graded
polynomials with arbitrary many alternating sets of variables which correspond
to the homogeneous components of $A$ and with a bounded number of
extra variables.

A key ingredient in the construction of these polynomials is
a presentation of any $G$-graded simple algebra as a tensor product of
two types of $G$-graded simple algebras, namely a twisted group algebra
(with fine grading) and a matrix algebra with an elementary grading.
Here is the precise statement. It is due to Bahturin, Sehgal and
Zaicev.

\begin{theorem} \cite{BSZ} \label{b}
Let $A$ be a finite dimensional $G$-graded simple algebra. Then there exists a
subgroup $H$ of $G$, a $2$-cocycle $\alpha :H\times H\rightarrow F^{*}$
where the action of $H$ on $F$ is trivial, an integer $r$ and an
$r$-tuple $(g_{1},g_{2},\ldots,g_{r})\in G^{r}$ such that $A$ is
$G$-graded isomorphic to $\Lambda=F^{\alpha}H\otimes M_{r}(F)$ where
$\Lambda_{g}=span_{F}\{\pi_{h}\otimes e_{i,j} \mid g=g_{i}^{-1}hg_{j}\}$.
Here $\pi_{h}\in F^{\alpha}H$ is a representative of $h\in H$ and
$e_{i,j}\in M_{r}(F)$ is the $(i,j)$ elementary matrix.

In particular the idempotents $1\otimes e_{i,j}$ as well as the
identity element of $A$ are homogeneous of degree $e\in
G$.
\end{theorem}

Let $t_1, \ldots, t_s>0$ be integers and, for $i=1, \ldots s$, define
$$
X_{g_i}^j=\{x_{1,g_i}^j,\ldots,
x_{m_i,g_i}^j\}\subseteq X_{g_i}, \ 1\le j\le  t_i,
$$
$t_i$ distinct
sets consisting of  $m_i\geq 0$ variables each of homogeneous degree $g_i$. Let also
$Y\subseteq \bigcup_{i=1}^s X_{g_i}$  be another set of homogeneous
variables disjoint from the previous sets.

Also, let $f=f(X^1_{g_1}, \ldots, X^{t_1}_{g_1}, \ldots, X^1_{g_s},
\ldots, X^{t_s}_{g_s},Y)\in F\langle X,G \rangle$ be a multilinear
graded polynomial in the variables from the sets $X^j_{g_i}$ and
$Y,$ $1\le i \le s$ and $1\le j \le t_i.$

This section is devoted to the proof of the following

\begin{theorem} \label{theo1}
Let $F$ be an algebraically closed field of characteristic zero
and $A$ a finite dimensional $G$-graded simple algebra over $F$.
For any $t\ge 1$, there exist integers
$$
2t\le t_1, \ldots, t_s \le 2t|G|
$$
and a $G$-graded polynomial

$$
f_t(X^{(t_{1},\ldots,t_{s})}_{G};Y)
= f_t(X^{1}_{g_1},\ldots, X^{t_1}_{g_1}, X^{1}_{g_2},
\ldots, X^{t_2}_{g_2},\ldots, X^{1}_{g_s},\ldots, X^{t_s}_{g_s}; Y)
$$
such that
\begin{enumerate}
\item
$f_t(X^{(t_{1},\ldots,t_{s})}_{G};Y)$ is not an identity of $A$; in particular it has an
evaluation in $A$ of the form $1\otimes e_{i,j}$.

\item
the cardinality of $Y$ depends on the order of $G$ and the
dimension of $A$ and not on the parameter $t$. In particular, the
cardinality of $Y$ is bounded.

\item
For every $j$ and $g\in G$, $|X^{j}_{g}|=\dim A_g$.

\item
$f_t(X^{(t_{1},\ldots ,t_{s})}_{G};Y)$  is alternating
on each one of the sets $X^{j}_{g}$.

\end{enumerate}
\end{theorem}

In view of the theorem above we claim that it is sufficient to
construct $G$-graded polynomials, which are non identities of $A$,
and correspond to the cyclic subgroups of $G$.

In order to make the statement precise, let $g$ be any element of
$G$ and let $S = \langle g\rangle$ be the subgroup it generates. We denote
by $d$ the order of $S$.

\begin{proposition}
It is sufficient to construct, for any integer $t\ge 1$,  a $G$-graded polynomial
(non identity of $A$)
$$
f_{t,g}(X_{S};Y_{S})
= f_{t,g}(X^{1}_e, \ldots, X^{2t}_e, X^{1}_g, \ldots, X^{2t}_g, \ldots, X^{1}_{g^{d-1}}, \ldots, X^{2t}_{g^{d-1}}; Y_{S})
$$
where
\begin{enumerate}

\item
$\mid Y_{S}\mid \leq r-1$ ($r$ is the cardinality of the tuple
which provides the elementary grading on $A$).

\item
for every $i=1,\ldots ,2t$, and every $0 \leq j \leq d-1$, we have that
$|X^{i}_{g^{j}}|=\dim A_{g^{j}}$.

\item

$f_{t,g}(X_{S};Y_{S})$ is alternating on the set
$X^{i}_{g^{j}}$ for every $i=1,\ldots,2t$, and every $0 \leq j \leq d-1$.

\item

$f_{t,g}(X_{S};Y_{S})$ admits a non-zero $G$-graded evaluation on $A$ of the form $\pi_{g^{l}}\otimes e_{1,r}$.

\end{enumerate}

\begin{remark}

Clearly, by adding an extra variable if necessary, we may assume
that the value of the polynomial above is $1 \otimes e_{1,1}$.

\end{remark}

\end{proposition}

\begin{proof}

Indeed, having constructed the polynomials $f_{t,g}=
f_{t,g}(X_{S};Y_{S})$ above, the required polynomials are given by

$$
f_{t}(X^{(t_{1},\ldots,t_{s})}_{G};Y) = \prod_{g \in G} f_{t,g}(X_{S};Y_{S}),
$$
where $\mid Y\mid = \sum_{g\in G}Y_{S}.$

\end{proof}

Consider the subalgebra $A_S=\oplus_{i=0}^{d-1} A_{g^i}\subseteq A$.
By \cite[Theorem 1.6]{quinn} (see also \cite[Theorem 18.13]{passman}), it is semisimple
and so it decomposes into the direct sum of $S$-graded simple algebras
$$
A_S \cong B_1\oplus B_2 \oplus \cdots \oplus B_l.
$$
It follows from Bahturin, Sehgal and Zaicev' result that for every $i=1,\ldots,l$, there exists a
subgroup $C_{i} \leq S$ and a $p_i$-tuple $(w_{i,1},\ldots,w_{i,p_i})$
of elements in $S$ (which determines the elementary grading on
$M_{p_i}(F)$) such that
$$
B_{i} \cong FC_{i}\otimes M_{p_i}(F)
$$
as $S$-graded algebras.

Notice that since $C_{i}$ is cyclic, $H^{2}(C_i, F^{*})=0,$  $1\le i\le l$.

The structure of $A_S$ is given here up to an $S$-graded
isomorphism. But we need more. In fact, in order to ``bridge" the
$S$-simple components $B_i$ by elements of $A$, we need to realize
the algebra $A_S$ as a subalgebra of $A$ in terms of its
presentation (i.e. with the terminology of Theorem \ref{b}). Here is
the precise statement.

First we make a definition: for a homogeneous subspace $D$ of
$A$ we define $weight(D) =\{g_i\in G \mid D\cap A_{g_i}\ne 0\}$.

\begin{proposition} \label{pro15}
With the above notation:
\begin{enumerate}

\item
For every $i=1,...,l$, $B_i \cong FC_i\otimes M_{p_i}(F)$ where $C_i=H^{g_{j_i}}\cap S$.
\item

In terms of the presentation of $A$, after a possible permutation
of the $r$-tuple $(g_1,\cdots,g_r)$ $($the tuple which provides
the elementary grading in the presentation of $A$$)$, we have
$$
B_i = \mbox{span} \{\pi_{h} \otimes e_{u,v} \in A_S\mid p_1+p_2+\cdots+p_{i-1}+1 \leq u,v \leq p_1+p_2+\cdots+p_{i}\}.
$$
In particular $p_1+p_2+\cdots+p_l=r.$

\item
For every $i=1,...,l$,
$$
B_i= B_{i,1}\oplus \cdots \oplus B_{i,c_{i}}
$$
(direct sum of vector spaces that we shall call ``pages")
where

\begin{enumerate}

\item
$B_{i,k} = \mbox{span} \{\pi_{h_{u,v}} \otimes e_{u,v} \in A_{S}\mid
p_1+p_2+\cdots+p_{i-1}+1 \leq u,v \leq p_1+p_2+\cdots+p_{i}\}$
(for any pair $(u,v)$ as above we choose a suitable $h=h_{u,v}
\in H$).

\item
Any homogeneous component of $B_{i}$ is ``concentrated'' in a
unique page. More precisely, $weight(B_{i,k}) \cap weight(B_{i,l})
=\emptyset$, if $k \neq l$.

\end{enumerate}
\end{enumerate}
\end{proposition}

Before proving the proposition let us show how to construct the
polynomial $f_{t,g}(X_{S};Y_{S})$.

For each ``page'' $B_{i,k}$ we construct a Regev polynomial
$f^{k}_{2p^{2}_i}$ (see \cite{for}), with $2p^{2}_i$ variables whose homogeneous
degrees are as the homogeneous degrees of $B_{i,k}$ and let

$$f_{i}=f^{1}_{2p^{2}_i}f^{2}_{2p^{2}_i}\cdots f^{c_{i}}_{2p^{2}_i}.$$

Now, the evaluation of $f^{k}_{2p^{2}_i}$ on a basis of $B_{i,k}$
gives an element of the form $\pi_{b_k} \otimes 1_{p_i \times p_i}$
where $\pi_{b_k}$ is a trivial unit of $FC_{i}$. This is a slight
abuse of notation since in fact the identity matrix $1_{p_i \times
p_i}$ is located in the block diagonal between rows $p_1+p_2+
\cdots + p_{i-1}+ 1$ and $p_1+p_2+ \cdots +p_{i}$.

\begin{remark}

This is the place where we use the fact that $C_i$ is a cyclic
group. Indeed, the evaluation of a Regev polynomial on the space
$B_{i,k}$ has the same effect as the evaluation on $p_i \times
p_i$ matrix algebra where all monomials are multiplied by the same
trivial unit which is obtained as the product of {\it commuting}
trivial units of $FC_{i}$.

\end{remark}

From the construction of $f_i$ we see that it has an evaluation of
the form $\pi_{b^{'}_i}\otimes 1_{p_i \times p_i}$ where
$\pi_{b^{'}_i}$ is a trivial unit of $FC_{i}$. Note that since the
homogeneous degrees of the spaces $B_{i,k}$ are disjoint for
different $k$, the polynomial $f_i$ is alternating (in
particular) on sets of $g^{\nu}$-variables, any $g^{\nu}\in S$, of
cardinality which is equal to the dimension of the
$g^{\nu}$-homogeneous subspace of $B_i$.

In order to get arbitrary many alternating sets we let $f^{t}_{i}$
be the product of $t$-copies (with disjoint sets of
variables) of $f_i$. Clearly, the evaluation of $f^{t}_{i}$ on a
basis of $B_{i}$ gives an element of the form $\pi_{a_i}\otimes
1_{p_i \times p_i}$ where $\pi_{a_i}$ is a trivial unit in $FC_i$.

Thus we have constructed a polynomial $f^{t}_{i}$ for any $B_i$.
We can now ``bridge" these polynomials and get a polynomial
$$
\phi_{S}=x_{0}f^{t}_{1}x_{1}f^{t}_{2}\cdots f^{t}_{l}x_{l+1}.
$$
We observe that with suitable evaluations on the $x$'s the polynomial $\phi_{S}$
has an evaluation which is equal to $1\otimes e_{1,1}$. But we are
not done yet. We still need to alternate variables with the same
homogeneous degrees in the different polynomials $f^{t}_{i}$ in
a suitable way. More precisely, for every $s=1,\ldots,t$, we alternate
all variables with the same homogeneous degrees which appear in
the polynomials $f^{s}_{1},f^{s}_{2}, \ldots, f^{s}_{l}$. Clearly,
because of the bridging variables $x_j$, the resulting
polynomial $f_{t,g}(X_{S};Y_{S})$ admits the value $1\otimes e_{1,1}$ and has
the required form.

Let us prove now the proposition above. To set up the notation
again we recall that $A$ has a presentation given by
$F^{\alpha}H\otimes M_{r}(F)$ and the elementary grading is given
by the $r$-th tuple $(g_1, g_2,\ldots,g_r)$. We let $S$ be the cyclic
group generated by $g \in G$ and we denote by $d$ its order. Let
us introduce an equivalence relation on the index set
$\{1,\ldots,r\}$ by setting $i \sim j$ if and only if
$g^{-1}_{i}Hg_{j} \cap S \neq \emptyset$. It is easy to see that
this is indeed an equivalence relation and we let
$$
\Omega_{1},\ldots,\Omega_{l}
$$
be the equivalence classes. We may clearly assume (after reordering the tuple $(g_1,
g_2,...,g_r)$) that equivalent indices are adjacent to each other.
In other words we have integers $p_{1},\ldots,p_{l}$ such that
$$
\Omega_{1}=\{1,\ldots,p_1\}, \Omega_{2}=\{p_{1}+1,\ldots,p_{1}+p_{2}\},
$$
$$
\ldots ,
\Omega_{l}=\{p_{1}+\cdots +p_{l-1}+1,\ldots,p_{1}+\cdots +p_{l}=r\}.
$$
We shall replace (as we may) elements from the
$r$-tuple which represent the same right $H$-coset by the same
representative.

Consider indices $i,j \in \Omega_{k}$. We know (by the definition
of the equivalence relation) that there exists an $h \in H$ such
that $g^{-1}_ihg_j \in S$. We claim that the number of elements $h
\in H$ with that property depends on the index $k$ but not on $i$
and $j$. In other words we claim the following.

\begin{lemma} \label{lem17}
For every $i,j, k_{0} \in \Omega_{k}$, $|g^{-1}_iHg_j \cap S| =
|g^{-1}_{k_0}Hg_{k_{0}} \cap S|$. Furthermore, the sets
$g^{-1}_iHg_j \cap S$ are $g^{-1}_{k_0}Hg_{k_{0}}\cap S$-cosets in
$S$.
\end{lemma}

\begin{proof}
Take $z \in g^{-1}_{i}Hg_{j} \cap S$. For different elements $q
\in g^{-1}_{j}Hg_{k_{0}} \cap S$, we obtain different elements $zq
\in g^{-1}_{i}Hg_{k_{0}} \cap S$ and hence $|g^{-1}_{i}Hg_{k_{0}}
\cap S| \geq |g^{-1}_{j}Hg_{k_{0}} \cap S|$. On the other hand,
taking inverses we see that $|g^{-1}_{j}Hg_{k_0} \cap S|=
|g^{-1}_{k_0}Hg_j \cap S|$. Being $i$, $j$ and $k_{0}$ arbitrary
the first part of the lemma follows. For the second part, note
that $g^{-1}_{i}Hg_{i} \cap S$ and $g^{-1}_{k_0}Hg_{k_{0}} \cap S$
are subgroups of the cyclic group $S$. By the first part of the
proof, they have the same order and hence they coincide. Following
the first part of the proof we see that $g^{-1}_iHg_j \cap S$ is a
(right, and hence $2$-sided (by commutativity)) $g^{-1}_iHg_i \cap
S$-coset and the lemma is proved.

\end{proof}

Now, observe that for $i=1,\ldots,l$, the algebra
$$
U_{i}=\mbox{span} \{\pi_{h} \otimes e_{u,v} \in A_{S}\mid u,v \in \Omega_{i}\}
$$
is a direct
summand of $A_{S}$ and so, the proof of the proposition will be
completed if we show that $U_{i}$, $i=1,\ldots,l$, is $S$-simple. To
see this we exhibit an explicit presentation of $U_i$ as a
$S$-simple algebra.

Fix $1 \leq k \leq l$ and $k_{0} \in \Omega_{k}$. Let
$w_{i-n_{(k-1)}} \in S$ be a $g^{-1}_{k_0}Hg_{k_{0}} \cap S$-coset
representative of $g^{-1}_{k_0}Hg_{i} \cap S$, where $p_1 + \cdots +
p_{k-1}+1\leq i \leq p_1 + \cdots + p_{k}$, and $n_{(k-1)}=p_1 +
\cdots + p_{k-1}$. Then the map
$$
\phi_{k}: \pi_{h}\otimes e_{i,j} \longmapsto w_{i-n_{(k-1)}}g^{-1}_ihg_jw^{-1}_{j-n_{(k-1)}} \otimes e_{i-n_{(k-1)}, j-n_{(k-1)}}
$$
determines an isomorphism of the $S$-graded algebra $U_{k}$ with
$F(g^{-1}_{k_0}Hg_{k_0}\cap S) \otimes M_{p_{k}}(F)$. In the latter,
the elementary grading is given by the $p_{k}$-tuple
$(w_1,\ldots,w_{p_k})$. Details are omitted.
Finally we note that Proposition \ref{pro15} (b) follows
from Lemma \ref{lem17} (by ordering the coset's elements) and the proof of
Proposition\ref{pro15} is completed.

\begin{remark}
Note that since the group $g^{-1}_{k_0}Hg_{k_0}\cap S$ is cyclic,
its cohomology vanishes and hence we may use group
elements rather then representatives.
\end{remark}

\section{Graded exponent}

Throughout $F$ will be an algebraically closed field of characteristic zero
and $A$ a $G$-graded PI-algebra over $F$ with  $G=\{g_1=1,g_2, \ldots, g_s\}$
a finite group.

In this section we shall prove that the $G$-exponent of $A$ exists and is an integer.
This result was proved in case $G$ is an abelian group in
\cite{AGL} for finitely generated algebras and in \cite{GL} in general.
Here we shall follow closely that approach.

We start by recalling the general setting.
The ideal of $G$-graded polynomial identities of $A$ is denoted $\mbox{Id}^G(A)$.
For every $n\geq 1$,
$$
P_n^G=\mbox{span}_F\{x_{\sigma(1),g_{i_{\sigma(1)}}}\cdots x_{\sigma(n),
g_{i_{\sigma(n)}}}\ | \ \sigma \in S_n, g_{i_1}, \ldots, g_{i_n}\in G \}
$$
is the space of multilinear $G$-graded  polynomials in the homogeneous
variables $x_{1,g_{i_1}}, \ldots, x_{n,g_{i_n}},$ $g_{i_j}\in G.$
We construct the quotient
space
$$
P_n^G(A)= \frac{P^G_n}{P^G_n\cap \I^G(A)}
$$
and the non-negative
integer
$$
c^G_n(A)=\dim_FP_n^G(A), \ n\geq 1,
$$
is the $n$th $G$-graded codimension of $A$.
Our aim is to prove that $exp^G(A)= \lim_{n\to \infty}\sqrt[n]{c_n^G(A)}$ exists and is an integer.
Moreover we shall relate such an integer to the dimension of a finite
dimensional semisimple algebra related to $A$.

For every $n\ge 1$, write $n=n_1+\cdots+n_s$ a sum of non-negative
integers and let $P_{n_1, \ldots, n_s}\subseteq P^G_n$ be the
space of multilinear graded polynomials in which the first $n_1$
variables have homogeneous degree $g_1,$ the next $n_2$ variables
have homogeneous degree $g_2$ and so on. Then $P_n^G$ is
the direct sum of subspaces isomorphic to $P_{n_1, \ldots, n_s}$,
for every choice of the integers $n_1, \ldots, n_s$. Moreover such
decomposition is inherited by $P_{n_1, \ldots, n_s}\cap \I^G(A)$ and we define
$$
P_{n_1, \ldots, n_s}(A)=\frac{P_{n_1, \ldots, n_s}}{P_{n_1, \ldots,
n_s}\cap \I^G(A)}.
$$
If we let
$$
c_{n_1,\ldots, n_s}(A)=\dim P_{n_1, \ldots, n_s}(A)
$$
then, by checking dimensions we have
\begin{equation}\label{coG}
c_n^G(A)=\sum_{n_1+\cdots+n_s=n}\binom{n}{n_1, \ldots,
n_s}c_{n_1,\ldots, n_s}(A),
\end{equation}
where $\displaystyle\binom{n}{n_1, \ldots, n_s}=\frac{n!}{n_1!\cdots
n_s!}$\  denotes the multinomial coefficient.  In order to
compute an upper and a lower bound for $c_n^G(A)$, it is enough to
do so for $c_{n_1,\ldots, n_s}(A)$ and then apply (\ref{coG}).

The space $P_{n_1, \ldots, n_s}(A)$ is naturally endowed with a
structure of $S_{n_1}\times \cdots \times S_{n_s}$-module in the
following way. The group $S_{n_1}\times \cdots \times S_{n_s}$ acts
on the left on $P_{n_1, \ldots, n_s}$ by permuting the variables of
the same homogeneous degree; hence $S_{n_1}$ permutes the variables
of homogeneous degree $g_1$,  $S_{n_2}$  those of homogeneous degree
$g_2$, etc.. Since $\I^G(A)$ is invariant under this action,
$P_{n_1, \ldots, n_s}(A)$ inherits a structure of $S_{n_1}\times
\cdots \times S_{n_s}$-module and we denote by $\chi_{n_1,\ldots ,n_s}(A)$ its character.

If $\lambda$ is a partition of $n$, we write $\lambda\vdash n$.
It is well-known that there is a one-to-one correspondence between partitions
of $n$ and irreducible $S_n$-characters. Hence if $\lambda\vdash n$,
we denote by $\chi_\lambda$ the corresponding irreducible $S_n$-character.
Now, if $\lambda(1)\vdash n_1 ,\ldots, \lambda(s)\vdash n_s,$
are partitions, then we write $\langle \lambda
\rangle=(\lambda(1), \ldots, \lambda(s))\vdash (n_1, \ldots, n_s)$ and we say
that  $\langle \lambda
\rangle$ is a multipartition of $n=n_1+\cdots +n_s$.

Since $\mbox{char}\, F=0$, by complete reducibility
$\chi_{n_1,\ldots ,n_s}(A)$  can be written as a sum of irreducible
characters and let
\begin{equation}\label{multiplo}
\chi_{n_1,\ldots ,n_s}(A)=\sum_{\langle \lambda\rangle\vdash
n}m_{\langle \lambda\rangle}\chi_{\lambda(1)}\otimes\cdots \otimes
\chi_{\lambda(s)},
\end{equation}
where $\langle \lambda \rangle=(\lambda(1), \ldots,
\lambda(s))\vdash (n_1, \ldots, n_s)$ is a multipartition of
$n=n_1+\cdots+n_s$ and $m_{\langle \lambda\rangle}\ge 0$ is the
multiplicity of $\chi_{\lambda(1)}\otimes\cdots \otimes
\chi_{\lambda(s)}$ in $\chi_{n_1,\ldots, n_s}(A)$.

Our first objective is to prove that the multiplicities in (\ref{multiplo})
are polynomially bounded.

Let $E=\langle e_1, e_2, \ldots \mid e_ie_j=-e_je_i \rangle$ be the
infinite dimensional Grassmann algebra over $F$ and let $E=E_0\oplus
E_1$ be its standard $\mathbb{Z}_2$-grading. 
Now, if $B=\oplus_{(g, i)\in
G\times \mathbb{Z}_2}B_{(g,i)}$ is a $G\times \mathbb{Z}_2$-graded algebra, then
$B$ has an induced $\mathbb{Z}_2$-grading, $B=B_0\oplus B_1$, where $B_0= \oplus _{g\in
G}B_{(g,0)}$ and $B_1= \oplus _{g\in
G}B_{(g,1)},$ and an induced $G$-grading $B=\oplus_{g\in G}B_g$ where, for all  $g\in G,$ $B_g=B_{(g, 0)}\oplus B_{(g, 1)}.$
Then one defines the Grassmann envelope of $B$ as $E(B)=(B_0\otimes E_0) \oplus (B_1\otimes E_1)$.
Clearly $E(B)$ has a  $G$-grading given by $E(B)=\oplus_{g\in G}E(B)_g$, where $E(B)_g=
(B_{(g,0)}\otimes E_0)\oplus (B_{(g,1)}\otimes E_1)$.

As in the ordinary case, the Grassmann envelope is an important object. In fact
by a result of Aljadeff and Belov (\cite{AB}),
any variety of $G$-graded PI-algebras can be generated by the
Grassmann envelope of a suitable finite dimensional $G\times
\mathbb{Z}_2$-graded algebra.

Let $\mathcal{V}= \mbox{var}^G(A)$ denote the variety of $G$-graded algebras
generated by $A$ and define $\mathcal{V}^*=\{ B= G\times
\mathbb{Z}_2\mbox{-graded algebra such that} \ E(B)\in \mathcal{V}\}$.
Then $\mathcal{V}^*$ is a variety (see \cite[Theorem 3.7.5]{GZbook}).

Now, according to \cite[Corollary 1]{GL}, there exist integers $k\ge l\ge 0$ such that
in (\ref{multiplo}) $\lambda(1), \ldots, \lambda(s) \in H(k,l)$, where
$H(k,l)=\{\lambda\vdash n \mid \lambda_{k+1}\le l\}$ is
the infinite $k\times l$ hook.

Let $L$ be the relatively
free $G\times \mathbb{Z}_2$-graded algebra of $\mathcal{V}^*$ on the $(k+l)s$ graded generators
\begin{equation}\label{run}
u_{(g_i,0)}^j, \ v_{(g_i,1)}^p, \quad 1\le i\le s, 1\le j\le k, 1\le p\le l.
\end{equation}
Then $\mathcal{V}= \mbox{var}^G(E(L))$ (see for instance  \cite[Theorem 4.8.2]{GZbook}).

Since $L$ is a finitely generated $G\times \mathbb{Z}_2$-graded PI-algebra, by
\cite[Theorem 1.1]{AB} there exists a finite dimensional $G\times \mathbb{Z}_2$-graded algebra $W$
such that $\mbox{var}^{G\times \mathbb{Z}_2}(L)= \mbox{var}^{G\times \mathbb{Z}_2}(W)$.
Moreover $L$ is a homomorphic image of a relatively free graded algebra $T$
of such variety on $ks$ even generators and $ls$ odd generators.

The algebra $T$ can be constructed by ``generic" elements as follows: fix a basis
$\{a_1\ldots, a_m\}$ of $W$ of homogeneous elements. Let $\xi_i^{(t)}$, \
$1\le i\le m, 1\le t\le (k+l)s$ be commutative variables and define
$\xi_{(g_i,0)}^{(t)} = \sum a_{i_j}\otimes \xi_{i_j}^{(t)}$, \ $1\le t\le k$,
where the sum runs over all $i_j$ such that $a_{i_j}$ is of homogeneous degree
$(g_i,0)$.
Similarly define  $\xi_{(g_i,1)}^{(r)} = \sum a_{i_j}\otimes \xi_{i_j}^{(r)}$, \ $1\le r\le l$,
where the $a_{i_j}$ are of homogeneous degree $(g_i,1)$.
Then $T$ is generated by the ``generic" elements
\begin{equation}\label{run1}
\xi_{(g_i,0)}^{(t)}, \xi_{(g_i,1)}^{(r)},
\ 1\le i\le s, 1\le t\le k, 1\le r\le l.
\end{equation}
Denote by $L_n$ the subspace of $L$ spanned by all products $w_1\cdots w_i$, $1\le i\le n$, where
$w_1, \ldots, w_i$ run over the generators given in (\ref{run}).
Define $T_n$ accordingly on the relatively free generators given in (\ref{run1}).
Since $L$ is a homomorphic image of $T$, $\dim L_n\le \dim T_n$ and we compute an upper bound of
this last dimension.

Notice that every monomial of degree at most $n$ in the generic elements in (\ref{run1}),
belongs to $W\otimes F[\xi_i^{(t)}]_n$, where $F[\xi_i^{(t)}]_n$ is the subspace of
polynomials of degree at most $n$ in the commutative variables $\xi_i^{(t)},$ \
$1\le i\le m, 1\le t\le (k+l)s$.

Since $\dim F[\xi_i^{(t)}]_n = \binom{(k+l)sm+n}{n}\le (n+(k+l)s)^{(k+l)sm}$,
we get that
\begin{equation}\label{run2}
\dim L_n\le \dim T_n \le m(n+(k+l)s)^{(k+l)sm}\le Cn^t,
\end{equation}
for some constants $C,t$.

We are now ready to prove the following.

\begin{lemma} \label{rem1}
There exist constants $C,t$ such that for all $n\ge 1$, $m_{\langle
\lambda\rangle}\le Cn^t$ in (\ref{multiplo}).
\end{lemma}

\begin{proof}
Suppose that there exists $n$ and $\langle \lambda\rangle\vdash n$ such that
$m_{\langle \lambda\rangle}> Cn^t \ge \dim L_n$, where $C$ and $t$ are defined in (\ref{run2}).
Hence there exist $m_{\langle \lambda\rangle}=r$ irreducible $S_{n_1}\times\cdots\times
S_{n_s}$-modules $M_1, \ldots, M_r\in P_{n_1, \ldots, n_s}$ with character
$\chi_{\lambda(1)}\otimes\cdots \otimes\chi_{\lambda(s)}$ and
$(M_1 \oplus \cdots \oplus M_r)\cap \mbox{Id}^G(\mathcal{V})=0$.
Now, as in \cite[Lemma 4]{GL}, we may take $M_i=F(S_{n_1}\times\cdots \times S_{n_s})h_i$
where, by eventually adding some empty sets of variables, we may assume that each $h_i$
is a polynomial in the homogeneous sets of variables
$Y_{g_1}^j\ldots, Y_{g_s}^j, Z_{g_1}^p, \ldots, Z_{g_s}^p$, \ $i\le j\le k, 1\le p\le l$, and
$h_i$ is symmetric in each $Y_{g_t}^j$ and alternating in each $Z_{g_t}^p$.
Since $(M_1 \oplus \cdots \oplus M_r)\cap \mbox{Id}^G(\mathcal{V})=0$, for every choice of
$\alpha_1,\ldots, \alpha_r\in F$ not all zero, we have that
$h=\sum_{i=1}^r \alpha_ih_i \not\in \mbox{Id}^G(\mathcal{V})$.

Let $\tilde{}$ be the map defined in \cite[Section 5]{GL}.
Then, if we regard the variables of each set $Y_{g_t}^j$ as even variables and those of $Z_{g_t}^p$
as odd variables, we can construct the polynomials $\tilde h_i, \ 1\le i\le r$.
Then $\tilde h_i$ is symmetric on each $Y_{g_t}^j$ and on each $Z_{g_t}^p$.

For every $i$ and $j$, let $Y_{g_i}^j=\{y_{1,g_i}^j, \ldots, y_{m_i,g_i}^j\}$
and $Z_{g_i}^j=\{z_{1,g_i}^j, \ldots, z_{r_i,g_i}^j\}$.
Define $S$ to be the relatively free $G$-graded algebra of the variety $\mathcal{V}$ on the graded
generators
$$
\bar y_{p, g_i}^j, \ \bar z_{q,g_i}^j,  \quad 1\le i\le s, 1\le p\le m_i, 1\le q\le r_i, \ j=1,2,\ldots.
$$
Then the algebra $Q=(S\otimes E_0)\oplus (S\otimes E_1)$ has a natural
$G\times \mathbb{Z}_2$-grading and we can take its Grassmann envelope
$$
E(Q)=(S\otimes E_0\otimes E_0)\oplus (S\otimes E_1\otimes E_1)
\subseteq S\otimes (E_0\otimes E_0\otimes E_1\otimes E_1).
$$
Since $E_0\otimes E_0\otimes E_1\otimes E_1$ is commutative, $E(Q)$ and $S$ satisfy
the same $G$-graded identities. Hence $E(Q)\in \mathcal{V}$ and, so, $Q\in \mathcal{V}^*$.

Now in each polynomial $\tilde h_i$, $1\le i\le r$, we identify the variables in each set
$Y_{g_i}^j$  and in each set $Z_{g_i}^j$
and we let $\tilde h_i^\circ$ be the corresponding polynomial.
Since $\deg \tilde h_i=n$, under the evaluation
$$
\varphi(Y_{g_i}^j)= u_{(g_i,0)}^j,\  \varphi(Z_{g_i}^p)= v_{(g_i,1)}^p,
\quad 1\le i\le s, 1\le j\le k, 1\le p\le l,
$$
we have that $\varphi(\tilde h_i^\circ)\in L_n$, for all $1\le i\le r$.

Since by hypothesis  $r> \dim L_n$, there exist scalars $\alpha_1, \ldots, \alpha_r$
not all zero, such that $\varphi(\sum_{i=1}^r \alpha_i\tilde h_i^\circ)=0$ in $L$.
Recalling that $L$ is a relatively free algebra of $\mathcal{V}^*$, we obtain that
$\tilde h^\circ =\sum_{i=1}^r \alpha_i\tilde h_i^\circ\in
\mbox{Id}^{G\times \mathbb{Z}_2}(\mathcal{V}^*)$.
Now, $Q\in \mathcal{V}^*$ and, so,  $\tilde h^\circ \in \mbox{Id}^{G\times \mathbb{Z}_2}(Q)$.

If we consider the evaluation in $Q=(S\otimes E_0)\oplus (S\otimes E_1)$ defined by
$$
\varphi(Y_{g_i}^j)= \bar y_{1, g_i}^j\otimes \alpha_1^j +\cdots +
\bar y_{m_i, g_i}^j\otimes \alpha_{m_i}^j, \quad 1\le i\le s, 1\le j\le k,
$$
$$
\varphi(Z_{g_i}^{j'})=
\bar z_{1, g_i}^{j'}\otimes \beta_1^{j'}+\cdots + \bar z_{r_i, g_i}^{j'}\otimes \beta_{r_i}^{j'},
\quad 1\le i\le s, 1\le {j'}\le l,
$$
where $\alpha_t^j, \beta_t^{j'}$ are disjoint monomials of $E$ of the correct homogeneous degree
($\alpha_t^j$ of homogeneous degree $0$ and  $\beta_t^{j'}$ of homogeneous degree $1$),
 we get that $\varphi(\tilde h^\circ)=0$.
By computing explicitly we obtain  $0=\varphi(\tilde h^\circ)=\varphi'(\tilde{\tilde h})\otimes \gamma$
where $0\ne \gamma \in E$ and $\varphi'$ is an evaluation such that
$$
\varphi'(y_{p,g_i}^j)= \bar y_{p,g_i}^j,
\varphi'(z_{q,g_i}^{j'})= \bar z_{q,g_i}^{j'},
$$
with $1\le p\le m_i, 1\le q\le r_i, 1\le j\le k,
1\le {j'}\le l,
1\le i\le s.$
Since $\tilde{\tilde h}=h$, we obtain that $\varphi'(h)=\varphi'(\tilde{\tilde h})=0$.

Now, recall that the elements   $\bar y_{p, g_i}^j, \ \bar z_{q,g_i}^{j},$  $j\ge 1$,
generate the relatively free $G$-graded algebra of $\mathcal{V}$ of
countable rank. Hence  $\varphi'(h)=0$ says that $h= \sum_{i=1}^r \alpha_i h_i
\in \mbox{Id}^G(\mathcal{V})$, and this contradiction completes the proof.
\end{proof}

Next we shall prove the existence of the $G$-exponent of $A$.
Let $B$ be a finite dimensional $G\times
\mathbb{Z}_2$-graded algebra such that $\mbox{var}^G(A)=\mbox{var}^G(E(B))$.
 By the Wedderburn-Malcev theorem \cite{CR} and
the result in \cite{taft}, we can write $B=B_1\oplus \cdots \oplus B_k+J$, where
$B_1, \ldots, B_k$ are $G\times
\mathbb{Z}_2$-graded simple algebras and $J$ is the Jacobson radical of $B$.
Recall that according to \cite[Section 3]{GL}, a semisimple subalgebra
$D=D_1\oplus\cdots\oplus D_h$,
where $D_1, \ldots, D_h\in \{B_1, \ldots, B_k\}$ are distinct, is
admissible if $D_1JD_2J\cdots JD_h\ne 0$.
Then one defines
\begin{equation} \label{d}
p=p(B)= \mbox{max}\, (\dim D)
\end{equation}
 where $D$ runs over all
admissible subalgebras of $B$.

We shall prove that $p$ coincides with  the $G$-exponent of $A$.
In fact we have.

\begin{theorem}\label{main} \label{mainth} Let $B$ be a
finite dimensional $G\times \mathbb{Z}_2$-graded algebra over an algebraically closed field of characteristic zero. Then there
exist constants $C_1>0, C_2, k_1,k_2$ such that
$$
C_1n^{k_1}p^n \le c^G_n(E(B)) \le C_2n^{k_2}p^n,
$$
where $p=p(B)$ is the integer defined in (\ref{d}).
\end{theorem}

\begin{proof}
This theorem is proved in \cite{GL} for $G$ abelian but the proof carries over to the non abelian case
by making the following changes.

In the computation of the upper bound $c^G_n(E(B)) \le C_2n^{k_2}p^n$ one should use Lemma
\ref{rem1} above instead of \cite[Remark 1]{GL}.

For the computation of the lower bound $C_1n^{k_1}p^n \le c^G_n(E(B))$ one should use
Theorem \ref{theo1} instead of \cite[Lemma 18]{AGL}, concerning the construction of
multialternating polynomials of arbitrary large degree for finite dimensional $G$-graded simple algebras.
We should point out that while the polynomial constructed in \cite[Lemma 18]{AGL} depends on a parameter $t$,
the one constructed in Theorem \ref{theo1} depends on $s$ parameters $t_1, \ldots, t_s$
(each corresponding to a homogeneous component of the algebra) and these parameters are
squeezed between $2t$ and $2|G|t$. Then one notices
that in \cite[Section 6]{GL} the proofs are carried over for each homogeneous component separately.
This completes the proof of the theorem.
\end{proof}

Since graded codimensions do not change by extension of the base field, we get
the following.

 \begin{theorem}
Let $G$ be a finite group and $A$ a $G$-graded PI-algebra over any field $F$ of characteristic zero.
 Then $exp^G(A)=\lim_{n\to \infty}\sqrt[n]{c_n^G(A)}$ exists and is an integer.
 \end{theorem}

\end{document}